\documentclass[11pt]{amsart}

\newif\ifdraft\draftfalse

\usepackage{amssymb}
\usepackage{amsthm}
\usepackage{eucal}
\usepackage[english]{babel}
\usepackage{nicefrac}

\makeatletter
\def\@begintheorem#1#2[#3]{%
    \def\naam{#1}
  \deferred@thm@head{\the\thm@headfont \thm@indent
    \@ifempty{#1}{\let\thmname\@gobble}{\let\thmname\@iden}%
    \@ifempty{#2}{\let\thmnumber\@gobble}{\let\thmnumber\@iden}%
    \@ifempty{#3}{\let\thmnote\@gobble}{\let\thmnote\@iden}%
    \thm@swap\swappedhead\thmhead{#1}{#2}{#3}%
    \the\thm@headpunct
    \thmheadnl % possibly a newline.
    \hskip\thm@headsep
  }%
  \ignorespaces}
\makeatother

\newcommand{\kantlijndraft}[1]{\ifdraft\hspace{-\lastskip}%
\vadjust{\vspace{-1mm}\smash{\llap{{\tt #1}\hspace{8mm}}}\vspace{1mm}}\fi}

\def\voegToe#1#2#3{\immediate\write1{\string\newlabel{#1}{{#2}{#3}}}}

\newcommand{\thlabel}[1]{\voegToe{#1}{\naam\noexpand~\thetheorem}{\thepage}\kantlijndraft{#1}}

\makeatletter
\renewcommand{\label}[1]{\voegToe{#1}{\@currentlabel}{\thepage}\kantlijndraft{#1}}
\makeatother

\newtheorem*{mtheorem}{Main Theorem}
\newtheorem{theorem}{Theorem}[section]
\newtheorem{lemma}[theorem]{Lemma}
\newtheorem{corollary}[theorem]{Corollary}

\newtheorem{proposition}[theorem]{Proposition}
\newtheorem{example}[theorem]{Example}
\theoremstyle{definition}
\newtheorem{definition}[theorem]{Definition}
\theoremstyle{remark}

\numberwithin{equation}{section}

\newtheorem{claim2}{\sc Claim}



\newcommand{\cupdot}{\mathbin{\mathaccent\cdot\cup}}

\newcommand{\sse}{\subseteq}						%subset or equal [\sse]
\newcommand{\minus}{\backslash}						%minus (compliment) [\minus]
\newcommand{\Un}{\bigcup}							%union (bigcup) [\Un]
\newcommand{\un}{\cup}								%union (cup) [\un]
\newcommand{\Meet}{\bigcap}							%intersection (bigcap} [\Meet]
\newcommand{\meet}{\cap}							%intersection (cap} [\meet]
\newcommand{\es}{\varnothing}						%emptyset [\es]

\newcommand{\scr}[1]{\ensuremath{\mathcal{#1}}}

\newcommand{\al}{\alpha}

\newcommand{\ka}{\kappa}

\newcommand{\y}{\psi}

\def\cprime{$'$}

\def\sapirovskii{{\v{S}}apirovski{\u\i}}
\def\arhangelskii{Arhangel{\cprime}ski{\u\i}}

\begin{document}

\title{Cardinality Bounds Involving the Skew-$\lambda$ Lindel\"of Degree and its Variants}

%    Information for authors:
\author{N.A. Carlson}\address{Department of Mathematics, California Lutheran University, 60 W. Olsen Rd, MC 3750, Thousand Oaks, CA 91360 USA}
\email{ncarlson@callutheran.edu}

\author{J.R. Porter}
\address{Department of Mathematics, University of Kansas, 405 Snow Hall, 1460 Jayhawk Blvd, Lawrence, KS 66045-7523, USA}
\email{porter@ku.edu}

%%%% change the following:


\let\thefootnote\relax\footnotetext{Dedicated to Filippo Cammaroto, on the
occasion of his sixty-fifth birthday.}

\begin{abstract} 
We introduce a modified closing-off argument that results in several improved bounds for the cardinalities of Hausdorff and Urysohn spaces. These bounds involve the cardinal invariant $skL(X,\lambda)$, the skew-$\lambda$ Lindel\"of degree of a space $X$, where $\lambda$ is a cardinal. $skL(X,\lambda)$ is a weakening of the Lindel\"of degree and is defined as the least cardinal $\kappa$ such that if $\scr{U}$ is an open cover of $X$ then there exists $\scr{V}\in [\scr{U}]^{\leq\kappa}$ such that $|X\minus\un\scr{V}|<\lambda$. We show that if $X$ is Hausdorff then $|X|\leq 2^{skL(X,\lambda)t(X)\psi(X)}$, where $\lambda= 2^{t(X)\psi(X)}$. This improves the well-known \arhangelskii-~\sapirovskii~ bound $2^{L(X)t(X)\psi(X)}$ for the cardinality of a Hausdorff space $X$. We additionally define several variations of $skL(X,\lambda)$, establish other related cardinality bounds, and provide examples.
\end{abstract}

\maketitle

%%%%%%%%%%%%%%%%%%%% section  %%%%%%%%%%%%%%%%%%%%%%%%%

\section{Introduction and Preliminaries.} Nearly  five decades ago, Arhangel'ski\u i~\cite{arh1969} established an impressive result by showing that the cardinality of first countable, Lindel\" of Hausdorff spaces   is no larger  than that of the real numbers. More precisely, he proved that for a Hausdorff space $X$, $|X| \leq 2^{\chi(X)L(X)}$ -  the size of an arbitrary Hausdorff space is bounded above by an exponential of two properties - the global Lindel\" off degree property $L(X)$ and the local character property $\chi(X)$.  Many improvements of this cardinality bound inequality have been published since 1969; most have been presented in an excellent  survey paper by Hodel~\cite{Hodel2006}.
\vskip 2mm
In this paper, the property of Lindel\" of is extended by adding a new dimension that measures how close in cardinality a family of open sets is to covering a space.  This additional flexibility in the Lindel\" of degree is used to derive new cardinality bounds, and it provides additional insight in understanding why the cardinality bound works. 
The paper starts with the Main Theorem, a set theoretic result that develops a closing-off argument in a non-topological setting. The proof of the main theorem requires an interesting modification involving a Ramsey-theoretic argument. The Main Theorem is applied in three separate settings;  the conclusions in each of these settings improve several classical cardinality bounds.
\vskip 2mm

\emph{All spaces considered in this paper are Hausdorff}. We assume that the reader is familiar with the usual cardinal functions of $L$ (Lindel\" of), 
$\chi$ (character), $\psi$ (pseudo-character), $\psi_c$ (closed pseudo-character), and $t$ (tightness).  Our notation and terminology follow \cite{Engelking}
for general topological notions, \cite{Hod84,Hodel2006,Juhasz} for cardinal functions and \cite{por88} for H-closed spaces and H-closed extensions.

\vskip 2mm
For a space $X$, $\tau(X)$ is used to denote the topology of $X$.  For $A\subseteq X$ and $\kappa$ an infinite cardinal, let $[A]^{\leq\kappa} = \{B: B \subseteq A, |B| \leq \kappa \}$.
The set  $cl_\kappa A = \cup\{cl(B): B \in A]^{\leq\kappa}\}$ is called the \emph{$\kappa$-closure of $A$}; $A$ is \emph{$\kappa$-closed} if $cl_\kappa A = A$.  It is easy to show that $cl_\kappa A \subseteq clA$ and  $cl_\kappa A$ is $\kappa$-closed.
The set  $cl_\theta A = \{p \in X: clU \cap A \ne \es$ for $p \in U \in \tau(X) \}$ is called the $\theta$-\emph{closure} of  $A$;
$A$ is $\theta$-\emph{closed} if $cl_\theta(A)=A$. It is clear that $clA\sse cl_\theta(A)$; however it is not necessarily true that $cl_\theta(cl_\theta(A))=cl_\theta(A)$.

\vskip 2mm

The first cardinality bound we will improve is the famous \arhangelskii -\sapirovskii \hskip 1mm bound of $ |X| \leq 2^{L(X)\psi(X)t(X)}$ for a space $X$ (see 2.27 in \cite{Juhasz}).  This 1974 bound is a modification by  \sapirovskii \hskip 1mm  of the 1969  \arhangelskii \hskip 1mm bound.

\vskip 2mm

For $A \subseteq X$, the \emph{almost Lindel\" of degree} of $A$ relative to $X$, denoted by $aL(A,X)$, is the smallest infinite cardinal $\kappa$ such that for every open cover $\mathcal U$ of $A$ by sets open in $X$, there is a subfamily $ \mathcal V \in [\mathcal U]^{\leq \kappa}$ such that $A \subseteq \cup \{clU: U \in \mathcal V\}$.  The \emph{almost Lindel\" of degree} of $X$, denoted by $aL(X)$, is $aL(X,X)$.
The \emph{almost Lindel\" of degree relative to closed subsets of} $X$, denoted by $aL_c(X)$, is the supremum of the set $\{aL(C,X): C$ closed subset of $X\}$.
The next cardinality bound we will improve is the 1988 Bella-Cammaroto bound 
$2^{aL_c(X)t(X)\psi_c(X)}$ for the cardinality of a space $X$ (see \cite{BellaCammaroto1988}).

\vskip 2mm

For $A \subseteq X$, the \emph{weak Lindel\" of degree} of $A$ relative to $X$, denoted by $wL(A,X)$, is the smallest infinite cardinal $\kappa$ such that for every open cover $\mathcal U$ of $A$ by sets open in $X$, there is a subfamily $ \mathcal V \in [\mathcal U]^{\leq \kappa}$ such that $A \subseteq cl\left(\cup \mathcal V\right)$.  The \emph{weak Lindel\" of degree} of $X$, denoted by $wL(X)$, is $wL(X,X)$.
The \emph{weak Lindel\" of degree relative to closed subsets of} $X$, denoted by $wL_c(X)$, is the supremum of the set $\{wL(C,X): C$ closed subset of $X\}$.
The third cardinality bound we will improve is the 1993 Alas  bound of  $2^{wL_c(X)\chi(X)}$ for the cardinality of an Urysohn space $X$ (see \cite{Alas1993}).  As $wL(X)=wL_c(X)$ for a normal space $X$, it follows that $|X|\leq 2^{wL(X)\chi(X)}$ for normals spaces, a result first proved in 1978 by Bell, Ginsburg, and Woods \cite{BGW1978}. 

\vskip 2mm

A basic space that is useful as an example is the Kat\u etov H-closed extension of $\omega$, denoted as $\kappa \omega$. The extension $\kappa\omega$ of $\omega$ is both H-closed and Urysohn and is the Stone-\v Cech compactification of $\omega$ with a finer topology; $\tau(\kappa\omega)$ is the largest topology such that     $\tau(\beta\omega)$ is the topology generated by the base $\{cl_{\kappa\omega}U: U \in \tau(\kappa\omega)\}$ (see 4.8 in \cite{por88} for details).  It is known (see \cite{CCP2013, Hodel2006}  or straightforward to compute that $L(\beta\omega) = aL(\beta\omega) = aL_c(\beta\omega) = wL(\beta\omega) = wL_c(\beta\omega) = \omega$, 
$L(\kappa\omega) = 2^\frak c, aL(\kappa\omega) = wL(\kappa\omega) = wL_c(\kappa\omega) = \omega,  aL_c(\kappa\omega) = \frak c$, $t(\beta\omega) = \frak c, t(\kappa\omega) = \omega$, $\psi(\beta\omega) = \frak c, \psi(\kappa\omega) = \omega$, $\psi_c(\beta\omega) =  \psi_c(\kappa\omega) = \frak c$, and $\chi(\beta\omega) = \frak c, \chi(\kappa\omega) = \omega$

\vskip 2mm

We conclude this section with some  results used in the paper.

\vskip 2mm

\begin{proposition}\label{prop}
 Let $X$ be a space, $A\sse X$ a closed subset of $X$, and $\kappa$ an infinite  cardinal. Then 
 \begin{itemize}
\item[(a)] \cite{Hodel2006} $\psi_c(X) \leq \psi(X)\cdot L(X)$,
\item[(b)] \cite{Hodel2006} if $X$ is Urysohn,   $\psi_c(X) \leq \psi(X)\cdot aL_c(X)$,
\item[(c)] \cite{CCP2013} $t(X) \leq \kappa$ iff for all $B \subseteq X, cl_\kappa B = cl B$,
\item[(d)]  \cite{CCP2013}  if $t(X) \leq \kappa$, then $aL_\kappa(X) = aL_c(X)$, 
\item[(e)] \cite{Hod84} if $\psi_c(X) \leq \kappa$, then $|cl_\kappa A| \leq |A|^{\leq \kappa}$, and 
\item[(f)] \cite{BellaCammaroto1988} if $X$ is Urysohn and $\chi(X) \leq \kappa$, then $|cl_\theta A| \leq |A|^{\leq \kappa}$.
 \end{itemize}
\end{proposition}

%%%%%%%%%%%%%%%%%%%% section  %%%%%%%%%%%%%%%%%%%%%%%%%
\section{A closing-off argument.}

Recall that for a set $X$ and subsets ${\scr S},{\scr T} \subseteq {\scr P}(X)$, a function $f:{\scr S} \rightarrow {\scr T}$ is an \emph{operator} if for each $A, B \in {\scr S}$ with $A \subseteq B, A \subseteq f(A) \subseteq f(B)$ and a function $h: {\scr S} \rightarrow X$ is \emph{expansive} if for $A \in \scr S$, $A \subseteq h(A)$.

We give our Main Theorem below, a closing-off argument giving a bound on the size of a set $X$. Note that $X$ need not be endowed with a topology and so the argument is entirely set-theoretic. This is not unlike Theorems 3.1 and 3.3, for example, in Hodel's survey paper \cite{Hodel2006}.
\begin{mtheorem}\label{mthm}
Let $X$ be a set,  $\ka$ an infinite cardinal, and $c :[X]^{\leq \kappa} \rightarrow [X]^{\leq 2^{\kappa}}$  an operator.  For each $x \in X$,  let $\{V(x,\al) : \al < \ka\}$ be a collection of subsets of $X$ satisfying this property:\\

{\rm\bf{(E)}} (expansive function) there is an expansive function $h:[\{V(x, \alpha):x \in X, \alpha < \kappa\}]^{\leq \kappa}\rightarrow X:\mathcal S \mapsto h(\cup \mathcal S)$.\\

{\rm\bf{(C-S)}} (cover-separation condition) if $\varnothing \ne H \in [X]^{\leq2^{\kappa}}$ and  $c(A) \subseteq H$ for all $A \in [H]^{\leq \kappa}$, then for $q \not\in H$,  there exist $A \in H^{\leq \kappa}$ and a function $f :A\rightarrow \ka$ such that $|H\backslash  h(\bigcup_{x < A}V(x,f(x)))| < 2^\kappa$ and $q \not\in h(\bigcup_{x \in A} V(x,f(x)))$. \\

\noindent Then $|X| \leq 2^{\ka}$.
\end{mtheorem}

\begin{proof}
Let $L:\scr{P}(X)\to X$ be a choice function. Let $M:[X]^{\geq 2^\kappa}\to[X]^{2^\kappa}$ be a function such that $M(A)\subseteq A$ for all $A\in[X]^{\geq 2^\kappa}$. (That is, for every subset $A$ of $X$ of cardinality at least $2^\kappa$, $M(A)$ is subset of $A$ of cardinality $2^\kappa$). Define $N:\scr{P}(X)\to\{\{x\}:x\in X\}\cup[X]^{2^\kappa}$ by
$$N(A)=\begin{cases}
\{L(A)\}&A\in[X]^{<2^\kappa}\\
M(A)&A\in[X]^{\geq2^\kappa}
\end{cases}$$

\noindent Note that $N$ is function on $\scr{P}(X)$ that chooses a point out of a subset $A$ if $A\in[X]^{<2^\kappa}$ and chooses a subset of cardinality $2^\kappa$ if $A\in[X]^{\geq2^\kappa}$. We construct a sequence $\{H_\alpha:0\leq\alpha<\kappa^+\} $ of subsets of $X$ such that for $0\leq\alpha<\kappa^+$,
\begin{enumerate}

\item $H_0=N(\varnothing)$

\item if $H_\beta$ is defined for $\beta<\alpha$, let

$$K_{\alpha} = \bigcup_{\beta<\alpha} H_\beta\cup\bigcup\{N(X\backslash h(\bigcup_{x\in A}V(x,f(x)))):A\in[\bigcup_{\beta<\alpha} H_\beta]^{\leq\kappa},f:A\to\kappa\}.$$
\end{enumerate}
Define $H_{\alpha}$ by $H_\alpha= \bigcup \{c(B):B \in [K_{\alpha}]^{\leq \kappa}\}$. \\

\noindent Now, for each $A\in[\bigcup_{\beta<\alpha} H_\beta]^{\leq\kappa}$ there are at most $\kappa^\kappa=2^\kappa$ functions $f:A\to\kappa$. Since $\left|N\left(X\backslash h(\bigcup_{x\in A}V(x,f(x))\right)\right|\leq 2^\kappa,$
 it follows that $|K_{\alpha}| \leq 2^{\kappa}$.  For each $B \in [K_{\alpha}]^{\leq \kappa}$,
 $|c(B)| \leq 2^{\kappa}$.  It follows that $|H_{\alpha}| \leq 2^{\kappa}$. Let $H=\bigcup\{H_\alpha:\alpha<\kappa^+\}$. Then $|H|\leq 2^\kappa$. Note also that for $B \in [H]^{\leq \kappa}$, there exists $\alpha<\kappa^+$ such that $B\sse H_\alpha\sse K_{\alpha+1}$ and thus $c(B) \subseteq H_{\alpha+1}\sse H$.
 
\vskip 3mm

\noindent We wish to show that $|X|\leq 2^\kappa$. Suppose by way of contradiction that $|X|>2^\kappa$. Since $|H|\leq 2^\kappa$, it follows that $|X|>|H|$ and $|X\backslash H|=|X|$.\\
 
\noindent Fix $q\in X\backslash H$. By \textbf{(C-S)}, there exist $A_q \in [H]^{\leq \kappa}$ and a function $f_q : A_q \rightarrow \kappa$ such that $|H \backslash h(\bigcup_{x \in A_q}V(x,f_q(x)))| < 2^\kappa$ and $q \not\in h(\bigcup_{x \in A_q}V(x,f_q(x)))$. Thus $h(\bigcup_{x\in A_q}V(x,f_q(x)))\subseteq X\backslash\{q\}$. Also, there exists $\alpha_q<\kappa^+$ such that $A_q\in\left[\bigcup_{\beta<\alpha_q} H_\beta\right]^{\leq\kappa}$.\\

\noindent Now unfix $q\in X\backslash H$. Since $|X\backslash H|=|X|>2^\kappa\geq\left|[H]^{\leq\kappa}\right|$,  there exists $p\in X\backslash H$ and $Y\subseteq X\backslash H$ such that $|Y|=|X\backslash H|=|X|$ and $A_p=A_q$ for all $q\in Y$.   Since $|Y| > 2^{\kappa}$ there is $Z \subseteq Y$ such that $|Z| = |Y| > 2^{\kappa}$ and $f_r= f_q$ for all $r, q \in Z$. \\ 

\noindent Thus, for $r \in Z$,  $h(\bigcup_{x\in A_r}V(x,f_r(x)))\subseteq X\backslash\{q\}$ for all $q \in Z$.  That is,  $h(\bigcup_{x\in A_r}V(x,f_r(x)))\subseteq X\backslash Z$ or $Z \subseteq X \backslash h(\bigcup_{x\in A_r}V(x,f_r(x)))$. This implies $|X\backslash h(\Un_{x\in A_r}V(x,f_r(x)))|>2^k$ and thus
$$\left|N(X\backslash h(\Un_{x\in A_r}V(x,f_r(x)))\right|=2^k.$$

\noindent Since $A_r\subseteq\left[\bigcup_{\beta<\alpha_r} H_\beta\right]^{\leq\kappa}$, it follows that
$$N(X\backslash h(\bigcup_{x\in A_r}V(x,f_r(x))) \subseteq H_{\alpha_r +1}\backslash h(\bigcup_{x\in A_r}V(x,f_r(x)))\subseteq H\backslash h(\bigcup_{x\in A_r}V(x,f_r(x))).$$\\

\noindent Therefore,
$$2^\kappa=\left|N(X\backslash h(\bigcup_{x\in A_r}V(x,f_r(x)))\right|
\leq\left|H\backslash h(\bigcup_{x\in A_r}V(x,f_r(x))))\right|
<2^\kappa.$$
This is a contradiction. Thus, $|X|\leq 2^\kappa$.
\end{proof}

We see that in the above proof, like a traditional closing-off argument, a chain of sets $\{H_\alpha:\alpha<\kappa^+\}$ is inductively constructed and the union $H=\Un_{\beta<\alpha}H_\beta$ is formed where $|H|\leq 2^\kappa$. But the above proof fundamentally diverges from standard arguments in the way it is shown that $|X|\leq 2^\kappa$ once the construction of $\{H_\alpha:\alpha<\kappa^+\}$ is complete. In a standard argument it is usually shown that $X=H$, and thus $|X|=|H|\leq 2^\kappa$. This is done by supposing there exists a point $q\notin H$ and obtaining a contradiction. Yet, what ultimately needs to be shown is that $|X|=|H|$, not necessarily that $X=H$, and in the above proof only $|X|=|H|$ is shown. This is accomplished by supposing that $|X|>|H|$, using every point $q\notin H$ rather than just one, and obtaining a contradiction. A Ramsey-theoretic relation is used to obtain the homogeneous set $Z$. 

Furthermore, in a traditional argument the chain of sets $\{H_\alpha:\alpha<\kappa\}$ is constructed whereby at stage $\alpha$ a certain collection of points is added to $\Un_{\beta<\alpha}H_\beta$. In the above proof however a certain collection of subsets of size $2^\kappa$ as well as points is added to $\Un_{\beta<\alpha}H_\beta$. 

In this paper, the function $h$ will be the identity function or the operator $c$.
\vskip 5mm 
%%%%%%%%%%%%%%  section %%%%%%%%%%%%%%%%%%%%%%%%%%%%%%%%
\section{The skew-$\lambda$ Lindel\"of degree.}

\begin{definition}
Let $X$ be a space, $A\sse X$, and $\lambda$ an infinite cardinal. The \emph{skew-$\lambda$ Lindel\"of degree of $A$ in $X$}, denoted by $skL(A,X,\lambda)$, is the least infinite cardinal $\kappa$ such that for every cover $\scr{U}$ of $A$ by open sets in $X$ there exists $\scr{V}\in[\scr{U}]^{\leq\kappa}$ such that $\left|A\backslash\bigcup\scr{V}\right|<\lambda$. The \emph{skew-$\lambda$ Lindel\"of degree of $X$}, denoted by $skL(X,\lambda)$, is $skL(X,X, \lambda)$. 
\end{definition}

We note that $skL(X,\lambda)\leq L(X)$ for any space $X$ and infinite cardinal $\lambda$. Also, as $skL(X, \lambda)$ and $\lambda$ are both infinite, it follows that $skL(X, \aleph_0) = skL(X, \aleph_1)$. Proposition~\ref{cldhered}(a) demonstrates that the cardinal function $skL(X,\lambda)$ is hereditary on closed subsets for every cardinal $\lambda$, and \ref{cldhered}(b) and is decreasing  in terms of $\lambda$.  Observe that \ref{cldhered}(e) is an improvement of Proposition~\ref{prop}(a). Proposition~\ref{cldhered}(d) gives a straightforward example of a space $X$ such that $skL(X,\lambda) < L(X)$ for a particular cardinal $\lambda$. For an infinite discrete space $X$, $skL(X, \lambda) = \omega$ if $|X| < \lambda^+$ and $skL(X, \lambda) = |X|$ if $|X| \geq \lambda^+$. 

%%%%%%%%Prop 3.2

\begin{proposition}\label{cldhered}
 Let $X$ be a space, $C\sse X$ a closed subset of $X$, and $\lambda$ an infinite cardinal. Then 
 \begin{itemize}
\item[(a)] $skL(C,\lambda)\leq skL(C,X,\lambda)\leq skL(X,\lambda)$,
\item[(b)] if $\mu \leq \lambda$, then $skL(X,\lambda) \leq skL(X, \mu)$,  
\item[(c)] $skL(X, \omega)  = L(X) \geq skLX,\lambda) \geq skL(X,|X|^+) = \omega$,
\item[(d)] if $Y$ is discrete space such that $|Y| = \frak c$, then \newline $skL(Y \cupdot  \beta\omega, \frak c^+) = \omega < \frak c = L(Y \cupdot  \beta\omega)$, and
\item[(e)] $\psi_c(X) \leq \psi(X)\cdot skL(X, \psi(X)^+)$.
 \end{itemize}
\end{proposition}

\begin{proof}  The proof of the first inequality of (a) is straightforward.
For the  second inequality of  (a), let $\kappa=skL(X,\lambda)$ and $\scr{U}$ be a cover of $C$ by sets open in $X$. Then $\{X\minus C\}\un\scr{U}$ is an open cover of $X$. There exists $\scr{V}\sse\{X\minus C\}\Un\scr{U}$ such that $|\scr{V}|\leq\kappa$ and $|X\minus\Un\scr{V}|<\lambda$. Let $\scr{W}=\scr{V}\minus\{X\minus V\}$. Then $\scr{W}\sse\scr{U}$, $|\scr{W}|\leq\kappa$, and 
$$\left|C\minus\Un\scr{W}\right|=\left|C\minus\Un\scr{V}\right|\leq\left|X\minus\Un\scr{V}\right| <\lambda.$$
Thus, $skL(C,X,\lambda)\leq\kappa$.  The proofs of (b), (c), and (d) are straightforward. For (e), let  $p \in X$, there is a family $\mathcal B \subseteq  \tau(X)$ such that $\cap \mathcal B = \{p\}$ and $|\mathcal B| \leq \psi(X)$. For each $U \in \mathcal B$ and $q \in X \backslash U$, there is an open set $V_q$ such that $x \in V_q \subseteq U$  and $q \notin clV_q$.  As $skL(X\backslash U,\psi(X)^+) \leq skL(X,\psi(X)^+)$, there is a subfamily $\mathcal C \subseteq \mathcal B$ such that $|\mathcal C| \leq skL(X,\psi(X)^+)$ and $|(X\backslash U)\backslash \cup \cup \mathcal C| \leq \psi(X)$.  Let  $\mathcal D_U = \{V_a:a \in (X\backslash U)\backslash \cup \cup \mathcal C\} \cup \mathcal C$.  Then $\cap \{clV: V \in \mathcal D_U\} \subseteq X\backslash U$ and $|\mathcal D_U| \leq \psi(X)\cdot skL(X, \psi(X)^+)$.  Now $\mathcal D = \cup\{\mathcal D_U: U \in \mathcal B\}$ has the property that  $|\mathcal D| \leq  \psi(X)\cdot skL(X, \psi(X)^+) $ and $\cap\{clV: V \in \mathcal D\} = \{p\}$.  This completes the proof that $\psi_c(X) \leq \psi(X)\cdot skL(X, \psi(X)^+)$.
\end{proof}

We now present another example of a space $X$ such that $skL(X,\lambda) < L(X)$.  This example has no isolated points.

\begin{example}\label{exampleA} \rm{
We construct an example of a space $X$ such that $|X| = 2^{\frak c}$, $skL(X, {\frak c}^+) = \omega$ and  $L(X) = \frak{c}$.
Let $\tau$ be the usual topology on $ \omega^* = \beta \omega \backslash \omega$ and ${\mathcal B} = \{B_{\alpha}: \alpha < \frak{c}\}$ a base for $\omega^*$.  For $\alpha < \frak c$, we can inductively find sets $C_{\alpha}$ such that $C_{\alpha} \subseteq B_{\alpha}$, $|C_{\alpha}| = \frak c$,  and $\{C_{\alpha}: \alpha < \frak c\}$ is a pairwise disjoint family. There is a subset $E \subset \frak c$  such that $\{B_{\alpha}: \alpha \in E\}$ is a family of pairwise disjoint sets and $|E| = \frak c$ (find an almost disjoint family $\mathcal F \in [\omega]^{\omega}$ with $|\mathcal F| = \frak c$ and note that $\{cl_{\beta \omega}F\backslash \omega:F \in \mathcal F\}$ works).
Let $D =\omega^* \backslash  \cup\{C_{\alpha}: \alpha \in E\}$.  As $|D| = \frak c$ and $|B_{\alpha}| = 2^{\frak c}$ for each $\alpha < \frak c$, $D$ is dense in $\omega^*$.  Let $X$ be $\omega^*$ with the topology generated by $\tau(\omega^*) \cup \{D\}$.} The space $X$ is not compact but is H-closed.
To show that  ${skL(X, {\frak c}^+) = \omega}$, let   $\mathcal C$ be open cover of $X$.  We can assume that $\mathcal C \subseteq \mathcal B$ and $\mathcal C = \{B_{\alpha}: \alpha \in A\} \cup \{B_{\alpha} \cap D: \alpha \in B\}$ where $A \cup B \subseteq \frak c$.  As $\omega^*$ is compact, there are finite subsets $F_A \in [A]^{<\omega}$ and  $F_B \in [B]^{<\omega}$ such that $X = \cup \{B_{\alpha}: \alpha \in F_A \cup F_B\}$. It follows that $|X\backslash( \cup \{B_{\alpha}: \alpha \in F_A\} \cup \{B_\alpha \cap D: \alpha \in F_B\}| \leq |D| = \frak c$.  This complete the proof that $skL(X, {\frak c}^+) = \omega$.  
As $X$ has a base of size $\frak c$, 
$L(X) \leq \frak{c}$. The open cover $\{B_{\alpha}:\alpha \in E\} \cup \{D\}$ of $X$ has no proper subcover. It follows that $L(X) \geq \frak c$.  Hence $L(X) = \frak c$. \qed
\end{example}

To apply the Main Theorem, we need, for $A \subseteq X$, that $|clA| \leq |A|^\kappa$ where $\kappa \geq \psi_c(X)$ (cf. 1.1(e)).  Usually, this is a consequence of $\psi_c(X)\leq L(X)\psi(X)$ of 1.1(a).
 By \ref{cldhered}(b,e), we know that $\psi_c(X)\leq skL(X,\lambda)\psi(X)$ for a cardinal $\lambda \leq \psi(X)^+$.  However, for an arbitrary $\lambda$, $\psi_c(X)\leq skL(X,\lambda)\psi(X)$ is not necessarily true.  For example, using the Kat\v etov extension $\kappa \omega$  space $\omega$, we have that $skL(\kappa\omega, (2^{\frak c})^+) = \omega$  and $\psi(\kappa\omega) =  \omega$ and $\psi_c(\kappa\omega) = \frak c$, as noted in \S1.  We will use the following lemma to skirt around this obstacle.

\begin{lemma}\label{lemmaA}
Let $X$ be a space and $\lambda$ an infinite cardinal. Define $\kappa=skL(X,2^{\psi(X)\cdot\lambda})\psi(X)\cdot\lambda$. Then for all $x\in X$ there exists a family $\scr{V}$ of open sets such that $|\scr{V}|\leq\kappa$, $x\in\Meet\scr{V}$, and $\left|\Meet_{V\in\scr{V}}clV\right|\leq 2^\kappa$.
\end{lemma}

\begin{proof}
Let $x\in X$. As $\psi(X)\leq\kappa$, there exists a collection $\scr{U}$ of open sets such that $|\scr{U}|\leq\kappa$ and $\{x\}=\Meet\scr{U}$. As $X$ is Hausdorff, for all $U\in\scr{U}$ and for all $y\in X\minus U$, there exists an open set $V(y,U)$ containing $x$ such that $y\in X\minus cl(V(y,U))$. Thus for all $U\in\scr{U}$, $\{X\minus cl(V(y,U)):y\in X\minus U\}$ is an open cover of $X\minus U$.

Fix $U\in\scr{U}$. Since $X\minus U$ is closed, by \ref{cldhered}(a) we have
$$skL(X\minus U,X,2^{\psi(X)\cdot\lambda})\leq skL(X,2^{\psi(X)\cdot\lambda})\leq\kappa.$$
There exists $A_U\sse[X\minus U]^{\leq\kappa}$ such that
$$\left|(X\minus U)\minus\Un_{y\in A_U}(X\minus cl(V(y,U))\right|=\left|(X\minus U)\meet\Meet_{y\in A_U}cl(V(y,U))\right|<2^{\psi(X)\cdot\lambda}\leq 2^\kappa.$$
Thus,
$$\left|\Un_{U\in\scr{U}}\left((X\minus U)\meet\Meet_{y\in A_U}cl(V(y,U))\right)\right|\leq2^\kappa|\scr{U}|\leq 2^\kappa\cdot\kappa=2^\kappa.$$
Let $Z=\Un_{U\in\scr{U}}\left((X\minus U)\meet\Meet_{y\in A_U}cl(V(y,U))\right)$. Then $|Z|\leq 2^\kappa$. 

Let $\scr{V}=\{V(y,U):U\in\scr{U}, y\in A_U\}$. We show now that $\left(\Meet_{V\in\scr{V}}clV\right)\minus\{x\}\sse Z$. Let $w\in\left(\Meet_{V\in\scr{V}}clV\right)\minus\{x\}$.
As $w\neq x$ and $\{x\}=\Meet\scr{U}$, there exists $W\in\scr{U}$ such that $w\in X\minus W$. Furthermore, we have that $w\in\Meet_{y\in A_W}cl(V(y,W))$ and so
$$w\in X\minus W\meet\Meet_{y\in A_W}cl(V(y,W))\sse Z.$$
This shows that $\left(\Meet_{V\in\scr{V}}clV\right)\minus\{x\}\sse Z$ and thus $\left|\left(\Meet_{V\in\scr{V}}clV\right)\minus\{x\}\right|\leq 2^\kappa$. It follows that $\left|\Meet_{V\in\scr{V}}clV\right|\leq 2^\kappa$.
As $x\in V$ for all $V\in\scr{V}$ and $|\scr{V}|\leq |\scr{U}|\cdot\kappa\leq\kappa\cdot\kappa=\kappa$, the proof is complete.
\end{proof}

A $\theta$-\emph{network} for a space $X$ was defined in \cite{CPR2012} as a non-empty collection of subsets $\scr{N}$ of $X$ such that for every point $x$ in an open set $U$ there exists $N\in\scr{N}$ such that $x\in N\sse clU$. The $\theta$-\emph{network weight} of $X$, denoted by $nw_\theta(X)$, is the least cardinality of a $\theta$-network for $X$. It is straightforward to see that if $X$ is regular then $nw_\theta(X)=nw(X)$.

\begin{lemma}\label{lemmaB}
Let $X$ be a space, $\lambda$ an infinite cardinal, and define\\ $\kappa=skL(X,2^{\psi(X)\cdot\lambda})\psi(X)\cdot\lambda$. Then $|X|\leq nw_\theta(X)^\kappa$.
\end{lemma}

\begin{proof}
By \ref{lemmaA}, for all $x\in X$ there exists a family of open sets $\scr{V}_x=\{V(x,\alpha):\alpha<\kappa$ such that $x\in\Meet\scr{V}_x$ and $\left|\Meet _{\alpha<\kappa}cl(V(x,\alpha))\right|\leq 2^\kappa$.

Let $\scr{N}$ be a $\theta$-network for $X$ such that $|\scr{N}|=nw_\theta(X)$. For all $x\in X$ and $\alpha<\kappa$, there exists $N(x,\alpha)\in\scr{N}$ such that $x\in N(x,\alpha)\sse cl(V(x,\alpha))$. Thus, for all $x\in X$, 
$$x\in\Meet_{\alpha<\kappa}N(x,\alpha)\sse\Meet _{\alpha<\kappa}cl(V(x,\alpha)).$$
Define 
$$\scr{C}=\left\{\Meet_{N\in\scr{M}}N:\scr{M}\in[\scr{N}]^{\leq\kappa}\right\}$$
and observe that $|\scr{C}|\leq nw_\theta(X)^\kappa$. If for each $x\in X$ we set $C_x=\Meet_{\alpha<\kappa}N(x,\alpha)$, we see that $C_x\in\scr{C}$ and $x\in C_x\sse\Meet _{\alpha<\kappa}cl(V(x,\alpha))$. Furthermore, $|C_x|\leq\left|\Meet _{\alpha<\kappa}cl(V(x,\alpha))\right|\leq 2^\kappa$. 

Define $\scr{B}=\{C_x:x\in X\}\sse\scr{C}$ and note that $|\scr{B}|\leq|\scr{C}|\leq nw_\theta(X)^\kappa$. For all $B\in\scr{B}$, choose $x_B\in X$ such that $B=C_{x_B}$. We can re-write $\scr{B}=\{C_{x_B}:B\in\scr{B}$. Furthermore, since $\Un\scr{B}=X$, we see that $X=\Un\{C_{x_B}:B\in\scr{B}$ and thus
$$|X|\leq|\scr{B}|\cdot 2^\kappa\leq nw_\theta(X)^\kappa\cdot 2^\kappa=nw_\theta(X)^\kappa.$$
\end{proof}

\begin{lemma}\label{lemmaC}
Let $X$ be a space and set $\kappa=skL(X,2^{\psi(X)t(X)})\psi(X)t(X)$. Then $|X|\leq d(X)^\kappa$.
\end{lemma}
\begin{proof}
It was shown in 3.2 in \cite{CPR2012} that $nw_\theta(X)\leq d(X)^{t(X)}$. After letting $\lambda=t(X)$ and using \ref{lemmaB}, we see that
$$|X|\leq nw_\theta(X)^\kappa\leq\left(d(X)^\lambda\right)^\kappa=d(X)^\kappa.$$
\end{proof}

\begin{lemma}\label{lemmaD}
Let $X$ be a space, let $A\sse X$, and set\\ $\kappa=skL(X,2^{\psi(X)t(X)})\psi(X)t(X)$.  Then $|clA|\leq |A|^\kappa$.
\end{lemma}

\begin{proof}
First note that $\psi(clA)\leq\kappa$, $t(clA)\leq\kappa$, and $d(clA)\leq|A|$. Furthermore, by \ref{cldhered}(a), we have that 

$$skL(clA, 2^{\psi(X)t(X)})\leq skL(clA, X, 2^{\psi(X)t(X)})\leq skL(X,2^{\psi(X)t(X)})\leq\kappa.$$ Now use \ref{lemmaC} above, where $X$ in \ref{lemmaC} is $clA$.
\end{proof}

We now apply the Main Theorem to obtain a new cardinality bound for any space $X$ (Theorem~\ref{thmA} below). In view of the fact that $skL(X,\lambda)\leq L(X)$ for any space $X$ and cardinal $\lambda$, this result gives an improvement on the well-known \arhangelskii -\sapirovskii~bound $|X| \leq 2^{L(X)\psi(X)t(X)}$. 

%%%%%%%%%%%Thm 3.8

\begin{theorem}\label{thmA}
For any space $X$,  
$$|X|\leq 2^{skL(X,2^{\psi(X)t(X)})\psi(X)t(X)}.$$
\end{theorem}

\begin{proof}
Let $\kappa=skL(X,2^{\psi(X)t(X)})\psi(X)t(X)$. We use the Main Theorem where the operator $c$ is the closure operator $cl$ and the expansion function $h$ is the identity function. \ref{lemmaD} guarantees that $cl$ is an operator such that $|clA|\leq 2^\kappa$ whenever $A\in[X]^{\leq\kappa}$.  As $\psi(X)\leq\kappa$, for each $x\in X$ there exists a pseudo-base $\{V(x,\alpha):\alpha<\kappa\}$ at $x$.

We need to verify that condition \textbf{C-S} in the Main Theorem holds. Let $\varnothing \ne H \in [X]^{\leq2^{\kappa}}$ be such that $clA \subseteq H$ for all $A \in [H]^{\leq \kappa}$. Observe that the fact that $clA \subseteq H$ for all $A \in [H]^{\leq \kappa}$ implies that $cl_\kappa(H)\sse H$. Also, since $t(X)\leq\kappa$, by \ref{prop}(c) it follows that $cl_\kappa(H)=clH$. Therefore $clH\sse H$ and $H$ is a closed set. By \ref{cldhered}, $skL(H,X,2^{\psi(X)t(X)})\leq\kappa$.

Let $q\in X\minus H$. For all $x\in H$, there exists $\alpha_x<\kappa$ such that $q\in X\minus V(x,\alpha_x)$. Then $\{V(x,\alpha_x):x\in H\}$ is a cover of $H$ by sets open in $X$. As $skL(H,X,2^{\psi(X)t(X)})\leq\kappa$, there exists $A\in[H]^{\leq\kappa}$ such that 
$$\left|H\minus\Un_{x\in A}V(x,\alpha_x)\right|< 2^{\psi(X)t(X)}\leq 2^\kappa.$$
Define $f:A\to\kappa$ by $f(x)=\alpha_x$ and observe $q\notin\Un_{x\in A}V(x,f(x))$. This verifies the condition \textbf{C-S}. 

We conclude by the Main Theorem that $|X|\leq 2^\kappa$.
\end{proof}

%%%%%%%%%%%%%  example 3.9   %%%%%%%

\begin{example}\label{E3.9} \rm{  Let $Z$ be a discrete space of size $\aleph_1$.  Using a theorem of Easton\cite{Eas}, we show in a certain model of ZFC that $$2^{skL(Z,2^{\psi(Z)t(Z)})\psi(Z)t(Z)} < 2^{L(Z)\psi(Z)t(Z)}.$$
Note that $L(Z) = {\aleph_1}$, $\psi(Z)t(Z) = \aleph_0$, and $2^{L(Z)\psi(Z)t(Z)} = 2^{\aleph_1}.$  On the other hand, $2^{skL(Z,2^{\psi(Z)t(Z)})\psi(Z)t(Z)} 
= 2^{skL(Z,2^{\aleph_0})}$ and when $|Z| < 2^{\aleph_0}$, by \ref{cldhered}(c), $ 2^{swL(Z,2^{\aleph_0})} = 2^{\aleph_0}$.  By Easton's Theorem, there is a model of ZFC in which $2^{\aleph_0} = \aleph_2$ and  $2^{\aleph_1} = \aleph_3$ are true.  In this model of ZFC, as $|Z| < 2^{\aleph_0}$,  $2^{skL(Z,2^{\psi(Z)t(Z)})\psi(Z)t(Z)}  = 2^{\aleph_0}= \aleph_2$.  However, $2^{L(Z)\psi(Z)t(Z)} = 2^{\aleph_1} = \aleph_3.$ \qed}
\end{example}

Example \ref{E3.9} answers (at least consistently) a question asked by Paul Szeptycki during the 2015 Summer Conference on Topology and its Applications held in Galway, Ireland.

\vskip 3.5mm

Let $X$ be a space and $\lambda$ an infinite cardinal. $X$ is \emph{skew-$\lambda$ Lindel\" of} if $skL(X,\lambda) = \omega$. Thus, $X$ is skew-$\lambda$ Lindel\" of if for every open cover $\scr{U}$ of $X$ there exists $\scr{V}\in[\scr{U}]^{\leq\omega}$ such that $\left|X\minus\Un\scr{V}\right|<\lambda$. The following is an immediate consequence of \ref{thmA}.

\begin{corollary}
If $X$ is a skew-$\lambda$ Lindel\" of space then $|X|\leq2^{t(X)\psi(X)}$, where $\lambda=2^{t(X)\psi(X)}$.
\end{corollary}

%%%%%%%%%%% section marker %%%%%%%%%%%%%%%%%%%%%%%%%
\section{The skew-$\lambda$ almost Lindel\"of degree}

\begin{definition}
Let $X$ be a space, $A\subseteq X$, and $\lambda$ an infinite cardinal. 
\begin{itemize}
\item[(a)] The \emph{skew-$\lambda$ almost Lindel\"of degree of $A$ in $X$}, denoted by $saL(A,X,\lambda)$, is the least infinite cardinal $\kappa$ such that for every cover $\scr{U}$ of $A$ by open sets in $X$ there exists $\scr{V}\in[\scr{U}]^{\leq\kappa}$ such that $\left|A\backslash\bigcup_{V\in\scr{V}}cl V\right|<\lambda$. The \emph{skew-$\lambda$ almost Lindel\"of degree of $X$}, denoted by $saL(X,\lambda)$, is $saL(X,X,\lambda)$. 
\item[(b)] For a cardinal $\kappa$, the \emph{skew-$\lambda$ almost Lindel\"of degree of $X$ with respect to $\kappa-$closed sets}, denoted by $saL_{\kappa}(X,\lambda)$, is defined by
$$saL_{\kappa}(X,\lambda)=\sup\{saL(A,X,\lambda):A\textup{ is $\kappa-$closed}\}.$$ 
\item[(c)] The \emph{skew-$\lambda$ almost Lindel\"of degree of $X$ with respect to closed sets}, denoted by $saL_{c}(X,\lambda)$, is defined by 
$$saL_{c}(X,\lambda)=\sup\{saL(A,X,\lambda):A \hskip 1mm \textup{ is closed}\}.$$
\end{itemize}
\end{definition}

Note that as $saL(X, \lambda)$ and $\lambda$ are both infinite, it follows that $saL(X, \omega) = saL(X, \omega_1)$ and  $saL_c(X, \omega) = saL_c(X, \omega_1)$. Useful relationships between the above cardinal invariants are given in the following proposition. 

\begin{proposition}\label{propA}
 For any space $X$ and cardinals $\lambda,\kappa$,
\begin{itemize}
\item[(a)] $saL(X,\lambda)\leq saL_c(X,\lambda) \leq skL(X,\lambda)\leq L(X)$,
\item[(b)] $saL_c(X,\lambda) \leq saL_{\kappa}(X,\lambda)\leq aL_{\kappa}(X)$,
\item[(c)] $saL(X,\lambda)\leq aL(X)$ and $saL_c(X,\lambda)\leq aL_c(X)$,

\item[(d)]  $saL_c(X, \omega) = aL_c(X) \geq saL_c(sX,\lambda) \geq saL_c(X,|X|^+) = \omega$,
\item[(e)] If $t(X) \leq \kappa$, then $saL_{\kappa}(X,\lambda) = saL_{c}(X,\lambda)$,
\item[(f)] If $saL_{c}(X,\lambda)t(X) \leq \kappa$, then $saL_{\kappa}(X,\lambda) \leq \kappa$,
\item[(g)] If $\lambda\leq\gamma$, then $saL_\kappa(X,\gamma)\leq saL_\kappa(X,\lambda)$, and
\item[(h)] if $Y$ is discrete space such that $|Y| = \frak c$, then \newline $saL(Y \cupdot  \beta\omega, \frak c^+) = \omega < \frak c = aL(Y \cupdot  \beta\omega)$, 
\end{itemize}
\end{proposition}
\begin{proof}
The proofs of (a), (b), (c), (d), (g), and (h) are straightforward.  For (e), observe that if $t(X)\leq\kappa$ then for all $A\sse X$, $cl_\kappa(A)=cl(A)$ by \ref{prop}(c). If $B$ is $\kappa$-closed, then $B=cl_\kappa (B)=cl(B)$ and $B$ is closed. Thus,
\begin{align}
saL_{\kappa}(X,\lambda)&=\sup\{saL(A,X,\lambda):A\textup{ is $\kappa-$closed}\}\notag\\
&\leq\sup\{saL(A,X,\lambda):A \hskip 1mm \textup{ is closed}\}\notag\\
&=saL_{c}(X,\lambda).\notag
\end{align}
Now apply (b). For (f), apply \ref{prop}(e).
\end{proof}

The next Lemma is needed to insure that the hypothesis of the following Theorem is satisfied for some infinite $\kappa$.

\begin{lemma}\label{lemA} Let $X$ be a space.  $\{\kappa: \kappa \geq \psi_c(X), \kappa \geq saL_{\kappa}(X,\lambda)\} \ne \varnothing$.  
\end{lemma}
\begin{proof}   For $\kappa = t(X)\psi_c(X)sL_c(X,\omega_1) $, $\psi_c(X) \leq  \kappa$.  By  \ref{propA}(d,e), we have that $saL_{\kappa}(X,\lambda) \leq \kappa$ for every $\lambda$. In particular, $saL_{\kappa}(X,2^{\kappa}) \leq \kappa$. 

\end{proof}

\begin{theorem}\label{thmsaL}
 Let X be a Hausdorff space and $\kappa$ an infinite cardinal such that $\psi_c(X) \leq \kappa$  and $saL_{\kappa}(X,2^\kappa) \leq \kappa$. Then $|X| \leq 2^{\kappa}$.
\end{theorem}

\begin{proof} We apply the Main Theorem. For $A \subseteq X$, let $c(A) = cl_{\kappa}(A)$. As $\psi_c(X) \leq \kappa$, for each $x \in X$, there is a family of open neighborhoods $\{B(x,\alpha):\alpha < \kappa\}$ of $x$ such that $\{x\} =\Meet_{\alpha<\kappa}B(x,\alpha)= \bigcap_{\alpha < \kappa}clB(x,\alpha)$. Let $V(x,\alpha) = cl(B(x,\alpha))$. We verify the hypotheses of the Main Theorem are satisfied, where $c$ is the $\kappa$-closure operator $cl_\kappa$. (It is noted in \cite{CCP2013} that $cl_\kappa$ is an operator). First note by \ref{prop}(e) that $cl_\kappa:[X]^{\leq\kappa}\to [X]^{\leq 2^\kappa}$. Let $h$ be the identity function.

We verify condition \textbf{C-S}. Suppose that $\es\neq H\in [X]^{\leq 2^\kappa}$ is such that $cl_\kappa(A)\sse H$ for all $A\in[H]^{\leq\kappa}$. We show $H$ is $\kappa$-closed. If $A\in[H]^{\leq\kappa}$ then by assumption $cl_\kappa(A)\sse H$. Thus, $\Un\{cl_\kappa(A):A\in[H]^{\leq\kappa}\}\sse H$. But $clA=cl_\kappa(A)$ for $|A|\leq\kappa$ by section 2, fact (c) in \cite{CCP2013}. Thus $cl_\kappa H=\Un\{clA:A\in[H]^{\leq\kappa}\}\sse H$ and $H$ is $\kappa$-closed.

Let $q\in X\minus H$. For all $p\in H$, there exists $f(p)\in\kappa$ such that $q\notin V(p, f(p))$. $\{B(p,f(p)): p\in H\}$ is an open cover of the $\kappa$-closed set $H$. As $saL_\kappa(X,2^\kappa)\leq\kappa$, there exists $A\in [H]^{\leq\kappa}$ such that
$$\left|H\minus\Un_{p\in A}V(p,f(p))\right|<2^\kappa.$$
Note also that $q\notin\Un_{p\in A}V(p,f(p))$. This verifies condition \textbf{C-S} and we conclude $|X| \leq 2 ^{\kappa}$. 
\end{proof}

Another way of stating \ref{thmsaL} is that for a space $X$ and  infinite cardinal $\lambda$,  $|X| \leq 2^{\mu}$ for any infinite cardinal  $\mu \in \{\kappa: \kappa \geq \psi_c(X)saL_{\kappa}(X,\lambda)\} $.  The next result plays a key role in applications of \ref{thmsaL}.

\begin{proposition}\label{proA} Let $X$ be a space. Then  $saL_c(X,2^{\chi(X)})\psi_c(X)t(X) \in \{\kappa: \kappa \geq \psi_c(X)saL_{\kappa}(X,2^{\kappa})\}$.
  \end{proposition}

\begin{proof} Let  $\kappa = saL_c(X,2^{\chi(X)})\psi_c(X)t(X)$.  By 4.2(g), $saL_{\kappa}(X,2^{\kappa}) \leq saL_{\kappa}(X,2^{\chi(X)})$.  By 4.2(e), since $t(X) \leq \kappa$, $ saL_{\kappa}(X,2^{\chi(X)}) =  saL_c(X,2^{\chi(X)})$.  Thus, $saL_{\kappa}(X,2^{\kappa}) \leq \kappa$. 
\end{proof}

We now establish an improvement of the Bella-Cammaroto bound that 
$|X| \leq 2^{aL_c(X)t(X)\psi_c(X)}$ for  a space $X$. Observe by \ref{propA}(c) that $saL_c(X,2^{t(X)\psi_c(X)})\leq aL_c(X)$ and by \ref{propA}(g), $saL_c(X,2^{\chi(X)}) \leq saL_c(X,2^{\psi_c(X)t(X)}) $.  The following result is an immediate consequence of \ref{proA}.

\begin{corollary}\label{saLc}
 For a space $X$,  
$|X|\leq 2^{saL_c(X,2^{t(X)\psi_c(X)})t(X)\psi_c(X)}.$
\end{corollary}

%%%%%%%%%%%%%% example 4.7  %%%%%%%%%%%%

\begin{example}\label{E4.7} \rm{  Let $Z$ be a discrete space of size $\aleph_1$. As in \ref{E3.9},   we work in a model of ZFC in which  both $2^{\aleph_0} = \aleph_2$ and  $2^{\aleph_1} = \aleph_3$ are true and show that $Z$ satisfies 

$$2^{saL_c(Z,2^{t(Z)\psi_c(Z)})t(Z)\psi_c(Z)} < 2^{aL_c(Z)t(Z)\psi_c(Z)}.$$

\noindent As $aL_c(Z) = {\aleph_1}$  and $\psi_c(Z) = \aleph_0$, in this model, $2^{aL_c(Z)t(Z)\psi_c(Z)} = 2^{\aleph_1} = \aleph_3$. 
On the other hand, $2^{saL_c(Z,2^{t(Z)\psi_c(Z)})t(Z)\psi_c(Z)}  
= 2^{saL_c(Z,2^{\aleph_0})} $.  Again, in this model and using \ref{propA}(d)  when  $|Z| < 2^{\aleph_0}$, it follows that  $ 2^{saL_c(Z,2^{\aleph_0})} = 2^{\aleph_0} = \aleph_2$.   \qed}
\end{example}

 \noindent As $\psi_c(X)t(X) \leq \chi(X)$ for a space $X$, another consequence of \ref{proA} is: $|X| \leq 2^{saL(X,2^{\chi(X)}){\chi(X)}}$. As $saL(X,2^{\chi(X)}) \leq skL(X,2^{\chi(X)})$ by \ref{propA}(a), we also have that $|X| \leq 2^{skL(X,2^{\chi(X)}){\chi(X)}}$.  Compare this with the bound $|X|\leq 2^{skL(X,2^{\psi(X)t(X)})\psi(X)t(X)}$ obtained in \ref{thmA}.  These two results are variations of each other but neither one implies the other.

\vskip 2mm 
\noindent Let $X$ be a space and $\lambda$ an infinite cardinal. The \textit{skew-$\lambda$ almost Lindel\"of closed pseudo-character of $X$}, denoted by $saL_{\kappa}\psi_c(X,\lambda)$, is defined as $\min\{\ka : \ka \geq \y_{c}(X)saL_{\ka}(X,\lambda)\}$.  The following result is an immediate consequence of the \ref{thmsaL} and an improvement of Corollary 1 in \cite{CCP2013}.

\begin{corollary} \label{Cor1} If $X$ is a space, then $|X| \leq 2^{saL_{\kappa}\psi_c(X)}$.
\end{corollary}

\vskip 5mm

%%%%%%%%%%%%%%%%% section %%%%%%%%%%%%%%%%%%%%%%%%%%%%%

\section{The skew-$\lambda$ weak Lindel\"of degree}
\begin{definition} 
Let $X$ be a space, $A\subseteq X$, and $\lambda$ an infinite  cardinal. 
\begin{itemize}
\item[(a)] The \emph{skew-$\lambda$ weak-Lindel\"of degree of $A$ in $X$}, denoted by $swL(A,X,\lambda)$, is the least infinite cardinal $\kappa$ such that for every cover $\mathcal{U}$ of $A$ by open sets in $X$ there exists $\mathcal{V}\in[\mathcal{U}]^{\leq\kappa}$ such that $\left|A\backslash cl\bigcup\mathcal{V}\right|<\lambda$. The \emph{skew-$\lambda$ weak-Lindel\"of degree of $X$}, denoted by $swL(X,\lambda)$, is $swL(X,X,\lambda)$. 

\item[(b)] The \emph{skew-$\lambda$ weak-Lindel\"of degree of $X$ with respect to closed sets}, denoted by $swL_{c}(X,\lambda)$, is defined by
$$swL_{c}(X,\lambda)=\sup\{swL(A,X,\lambda):A \hskip 1mm \textup{ is closed}\}.$$
\end{itemize}
\end{definition}

Note that as $swL(X, \lambda)$ and $\lambda$ are both infinite, it follows that $swL(X, \omega) = swL(X, \omega_1)$ and  $swL_c(X, \omega) = swL_c(X, \omega_1)$. Useful relationships between the above cardinal invariants are given in the following proposition.

%%%%%%%%%%%%%%Prop 5.2

\begin{proposition}\label{prop5}
 For any space $X$ and cardinals $\lambda,\kappa$,
\begin{itemize}
\item[(a)] $swL(X,\lambda)\leq saL(X,\lambda)\leq saL_c(X,\lambda) \leq skL(X,\lambda)\leq L(X)$,
\item[(b)] $swL_c(X,\lambda) \leq saL_c(X,\lambda) \leq saL_{\kappa}(X,\lambda)\leq aL_{\kappa}(X)$,
\item[(c)] $swL(X,\lambda)\leq wL(X)$ and $swL_c(X,\lambda)\leq wL_c(X)$,
\item[(d)]  $swL_c(X, \omega) = wL_c(X) \geq swL_c(sX,\lambda) \geq swL_c(X,|X|^+) = \omega$,
\item[(e)] If $\lambda\leq\gamma$, then $swL_c(X,\gamma)\leq swL_c(X,\lambda)$, 
\item[(f)] If $X$ is normal, then $swL_c(X,\gamma) = swL(X,\lambda)$, and

\item[(g)] if $Y$ is discrete space such that $|Y| = \frak c$, then \newline $swL(Y \cupdot  \beta\omega, \frak c^+) = \omega < \frak c = wL(Y \cupdot  \beta\omega)$, 
\end{itemize}
\end{proposition}

\begin{proof}  The proofs are straightforward. 
\end{proof}

The following result provides an improvement of the  Alas bound  $|X| \leq 2^{wL_c(X)\chi(X)}$ for Urysohn space $X$ as $swL_c(X,\lambda)\leq wL_c(X)$ for any cardinal $\lambda$. 

%%%%%%%%%%Theorem 5.3

\begin{theorem}\label{thmUry}
If $X$ is Urysohn, then $|X|\leq 2^{swL_c(X,2^{\chi(X)})\chi(X)}$.  
\end{theorem}

\begin{proof} We apply the Main Theorem. Let $\kappa = swL_c(X,2^{\chi(X)})\chi(X)$. For $A \subseteq X$, let $c(A) = cl_{\theta}A$ and $h(A)= cl(A)$. By \ref{prop}(f), as $\chi(X) \leq \kappa$, $|cl_\theta(A)| \leq |A|^{\kappa}$; so, $c:[X]^{\leq\kappa}\to [X]^{\leq 2^\kappa}$. It is straightforward to show that the function $c$ is an operator and $h$ is expansive.  For all $x\in X$, let $\{V(x,\alpha):\alpha<\kappa\}$ be a neighborhood base of open sets at $x$. We need to verify condition {\bf{C-S}}. Suppose that $\es\neq H\in [X]^{\leq 2^\kappa}$ is such that $cl_\theta(A)\sse H$ for all $A\in[H]^{\leq\kappa}$. First, we show $H$ is $\theta$-closed.  Let $p \in cl_\theta H$, $x_\alpha \in cl(V(p,\alpha) \cap H$ for $\alpha \in \kappa$, and $A  = \{x_\alpha: \alpha \in \kappa\}$. Note that $p \in cl_\theta A$ and $A \in [H]^{\leq\kappa}$.  As $cl_\theta A \subseteq H$, it follows that $H$ is $\theta$-closed.
To complete the verification of \textbf{C-S}, we start with a point $q\in X\minus H$. There is some $\beta \in \kappa$ such that $cl(V(q,\beta)) \cap H = \es$.   For each $p\in H$, as $\{V(p,\alpha):\alpha<\kappa\}$ is a base, there exists $f(p)\in\kappa$ such that $V(q,\beta)\cap V(p, f(p)) = \es$. $\{V(p,f(p)): p\in H\}$ is an open cover in $X$ of the $\theta$-closed set $H$. As $saL_\kappa(X,2^\kappa)\leq\kappa$, there exists $A\in [H]^{\leq\kappa}$ such that
$$\left|H\minus cl\left(\Un_{p\in A}V(p,f(p))\right)\right|<2^\kappa.$$
Note also that $q\notin cl(\Un_{p\in A}V(p,f(p)))$. This completes the verification of the condition \textbf{C-S} and we conclude $|X| \leq 2 ^{\kappa}$. 
\end{proof}

%%%%%%%%%%%%% example 5.4 %%%%%%%%%%%%%%%

\begin{example}\label{E5.4} \rm{ As in \ref{E3.9} and \ref{E4.7}, we will show that   a discrete space $Z$ of size $\aleph_1$, in a model of ZFC in which  both $2^{\aleph_0} = \aleph_2$ and  $2^{\aleph_1} = \aleph_3$ are true, satisfies 

$$2^{swL_c(Z,2^{\chi(Z)})\chi(Z)} < 2^{wL_c(Z)\chi(Z)}.$$

\noindent First note that $Z$ is Urysohn.  As $wL_c(Z) = {\aleph_1}$  and $\chi(Z) = \aleph_0$, in this model, $2^{wL_c(Z)\chi(Z)} = 2^{\aleph_1} = \aleph_3$. 
On the other hand, $2^{swL_c(Z,2^{\chi(Z)})\chi(Z)}   
= 2^{swL_c(Z,2^{\aleph_0})} $.  In this model, using \ref{prop5}(d) as  $|Z| < 2^{\aleph_0}$, $ 2^{swL_c(Z,2^{\aleph_0})} = 2^{\aleph_0} = \aleph_2$.   \qed}
\end{example}  

In light of \ref{prop5}(f), we have the following corollary.
\begin{corollary}
If $X$ is normal then $|X|\leq 2^{swL(X,2^{\chi(X)})\chi(X)}$.
\end{corollary}

%%%%%%%%%%section %%%%%%%%%%%%%%

\section{Open Problems}

In Examples, \ref{E3.9}, \ref{E4.7}. and \ref{E5.4}, we used a model of ZFC to show that the three major cardinality inequalities developed  in this paper actually improve the known corresponding cardinality inequalities.   
We had hope that the Kat\v etov extension $\kappa\omega$ of $\omega$ would be the ZFC space $X$ such that  $$2^{skL(X,2^{\psi(X)t(X)})\psi(X)t(X)}  < 2^{L(X)\psi(X)t(X)}.$$ Our motivation being that $2^{L(\kappa\omega)\psi(\kappa\omega)t(\kappa\omega)} = 2^{2^{\frak c}}$ whereas $|\kappa\omega| = 2^{\frak c}$.  However, $2^{skL(\kappa\omega,2^{\psi(\kappa\omega)t(\kappa\omega)})\psi(\kappa\omega)t(\kappa\omega)} = 2^{skL(\kappa\omega,2^{\aleph_0})} = 2^{2^{\frak c}}$  is also true. That is, both cardinality bounds are excessive.  We conclude this paper with the problem of finding ZFC spaces that show that  the  cardinality inequalities of  \ref{E3.9}, \ref{E4.7}, and \ref{E5.4} can be strict.


\begin{thebibliography}{1}
%%%%% after done writing paper, throw out unused references
\bibitem{arh1969} 
A.~V. ArhangelÕskiùõ, \emph{On the cardinality of bicompacta satisfying the first axiom of countability}, Soviet Math. Dokl. 10 (1969) 951--955.
\bibitem{Alas1993}
O.~T. Alas, \emph{More topological cardinal inequalities}, Colloq. Math. \textbf{65} (1993), 165--168.
\bibitem{BGW1978}
M.~Bell, J.~Ginsburg, G.~Woods, \emph{Cardinal inequalities for topological spaces involving the weak Lindel\"of number}, Pacific Journal of Mathematics \textbf{79} (1978), no.~1, 37--45.
\bibitem{BellaCammaroto1988}A.~Bella, F.~ Cammaroto, \emph{On the cardinality of Urysohn spaces}, Canad. Math. Bull. 31 (1988) 153Ð158.
\bibitem{CCP2013}F.~Cammaroto, A.~Catalioto, J.R~Porter, \emph{On the cardinality of Hausdorff spaces}, Topology Appl.\textbf{160} (2013), no.~1, 137--142. 
\bibitem{CPR2012}N.~Carlson, J.R.~Porter, G.J.~Ridderbos, \emph{On Cardinality Bounds for Homogeneous Spaces and the $G_\kappa$-modification of a Space}, Topology Appl. \textbf{159}, (2012), no.~13, 2932--2941. 
\bibitem{DowPorter82}A.~Dow, J.R.~Porter, \emph{Cardinalities of H-closed spaces}, Topology Proc. \textbf{7} (1982), no. 1, 27--50. 

\bibitem{Eas} W.~Easton, \emph{Powers of regular cardinals},  Ann. Math. Logic \textbf{1} (1970), no. 2, 139--178.


\bibitem{Engelking} R. Engelking, \emph{General {T}opology}, Heldermann Verlag, Berlin, second ed., 1989.
\bibitem{Hod84} R.E.~Hodel, {Cardinal functions I}, in: K. Kunen, J.E. Vaughan (Eds.), Handbook of Set-Theoretic Topology, Elsevier North-Holland, Amsterdam, 1984, pp. 1Ð61.
\bibitem{Hodel2006}R.E.~Hodel, \emph{Arhangel'skii's solution to Alexandroff's problem: A survey}, Top. and its Appl. 153 (2006) 2199--2217.
\bibitem{Juhasz}
I.~Juh{\'a}sz, \emph{Cardinal functions in topology---ten years later}, second
  ed., Mathematical Centre Tracts, vol. 123, Mathematisch Centrum, Amsterdam,
  1980.
\bibitem{por88}J.~R. Porter, G.~Woods, \emph{Extensions and Absolutes of Hausdorff Spaces}, Springer,  Berlin, 1988.
\end{thebibliography}
\end{document}